\documentclass[12pt]{amsart} 

\usepackage{a4wide}
\usepackage{amsmath, amssymb, amsfonts, amsthm}

\newcommand{\R}{\mathbb{R}}
\newcommand{\Z}{\mathbb{Z}}
\newcommand{\eps}{\varepsilon}

\DeclareMathOperator{\Id}{Id}
\DeclareMathOperator{\dist}{dist}
\DeclareMathOperator{\spn}{span}
\DeclareMathOperator{\supp}{supp}

\newtheorem{theorem}{Theorem}
\newtheorem{fact}{Fact}
\newtheorem{lemma}{Lemma}
\newtheorem{corollary}{Corollary}[theorem]

\title[Stability of nearly optimal decompositions]{Stability of nearly optimal decompositions in Fourier Analysis}

\author{Anton Tselishchev}

\thanks{This research was supported by the Russian Science Foundation (grant No.~18-11-00053).}

\address{Chebyshev Laboratory, St. Petersburg State University, 14th Line V.O., 29B, Saint Petersburg 199178 Russia}
\address{St. Petersburg Department of Steklov Mathematical Institute, Fontanka 27, St. Petersburg 191023, Russia}
\email{celis-anton@yandex.ru}

\date{}

\begin{document}

\begin{abstract}
The question of existence is treated for near-minimizers for the distance functional (or $E$-functional in the interpolation terminology) that are stable under the action of certain operators. In particular, stable near-minimizers for the couple $(L^1, L^p)$ are shown to exist when the operator is the projection on wavelets and these wavelets possess only some weak conditions of decay at infinity.
\end{abstract}

\maketitle

\section{Introduction}
Let $(X, Y)$ be a couple of Banach spaces and $f\in X$. Consider the distance functional from $f$ to the ball of radius $s$ in 
$Y$:
$$
E(s, f; X, Y)=\dist_X (f, B_Y(s))=\inf\{\|f-g\|_X: \|g\|_Y\leq s\}.
$$
In the book \cite{KK} the near-minimizers for this functional (and some other functionals) are studied. By this we mean functions $g$ such that
$$
\|g\|_Y\leq Cs\ \  \hbox{and}\ \  \|f-g\|_X\leq C\dist_X\big(f, B_Y\big(\frac{s}{C}\big)\big).
$$
We are interested in the behaviour of near-minimizers under the action of certain operators $T$. It is clear that if 
$T$ is bounded on $X$ and $Y$ then $Tg$ will also belong to the ball of radius nearly $s$ in $Y$ 
(which means that $\|Tg\|_Y\leq Cs$) and $\|Tf-Tg\|_X\leq C\dist (f, B_Y(\frac{s}{C}))$ (here $C$ stands for some other constant). In particular if $\dist_X(f, B_Y(t))\leq C\dist_X(Tf, B_Y(t))$ then $Tg$ will be a near-minimizer for $Tf$. 

In this regard, we will be interested in operators which are unbounded on $X$ --- can we say something about their action on near-minimizers? The corresponding stability theorems are helpful in reducing the problems of evaluation of various functionals in interpolation theory (and thus the interpolation spaces) for complicated pairs of Banach spaces to the case of more 
simple embracing pairs. Stable near-minimizers for $K$-functionals are the most effective tools for these problems, cf., for example the "shift of smoothness" theorem in \S 10.2.2. in the book \cite{KK}. In this article, however, we study the more
"demonstrative" distance functional (or $E$-functional in the interpolation terminology). However, the problems about near-minimizers for $E$- and $K$- functional can in a sense be reduced to one another --- cf. \S 5.4. in \cite{KK}. 

In the book \cite{KK} $T$ usually stands for a Calder\' on--Zygmund operator and $X$ --- for the space $L^1$. As for the space 
$Y$, the $L^p$ spaces with $1<p<\infty$, $L^\infty$ or (homogeneous) Campanato spaces $\dot{C}_p^{s,k}$ are considered. The minimizers which are in a sense "stable" under the action of $T$ are constructed there. The essential ingredients of these
constructions are the Calder\' on--Zygmund decompositions or its smooth analogues.

One of the statements proved in that book is the following. 

\begin{theorem}
Let $T$ be a Calder\' on --Zygmund operator and $f\in L^1$ is a function for which $Tf\in L^1$. Then for any $s>0$ there exists
such function $u^{(s)}\in L^1$ that the following conditions hold:
\begin{align*}
   \|u^{(s)}\|_{L^p}\lesssim& s,\\
   \|f-u^{(s)}\|_{L^1}\lesssim& \dist_{L^1}(f,B_{L^p}(s)),\\
   \|Tf-Tu^{(s)}\|_{L^1}\lesssim& \dist_{L^1}(f,B_{L^p}(s))+\dist_{L^1}(Tf,B_{L^p}(s)).
  \end{align*}
\end{theorem}

Here we say that $A\lesssim B$ if $A\leq CB$ for some constant $C$. It will always be clear from the context from which parameters $C$ can depend and from which it can not (or it will be stated explicitly). Here these constants do not depend on $s$ and $f$.

The first two conditions in this theorem mean that $u^{(s)}$ is a near-minimizer for the distance functional for $f$ at $s$ and the third one says that $Tu^{(s)}$ behaves much like the near-minimizer for the distance functional for $Tf$ at $s$ (in particular, it will be the near-minimizer if the second term majorizes the first one).

One of the proofs presented in the book reproduces Bourgain's arguement from paper \cite{Bo} --- the arbitrary near-minimizer is turned to the stable one by adding the summand which is a "good" part of Calder\' on--Zygmund decomposition of a certain function.

We will be intersted in stability of near-minimizers in some cases that are not treated in the book \cite{KK} --- more precisely, when operator $T$ is a projection on wavelets which possess only some weak conditions of decay at infinity (in this case $T$ might not be the singular integral operator in the classic sense) or when $T$ is a usual singular intagral operator but $X$ and $Y$ are weighted $L^1$ and $L^p$ spaces. The proofs will also use the Bourgain's arguments but instead of the standard Calder\' on--Zygmund decomposition of a function into the "bad" and "good" parts some other suitable decompositions will be useful.

The author is kindly greatful to his scientific advisor, S. V. Kislyakov, for posing these problems and for the continuous support during the process of their solutions.

\section{The stability theorem for projections on wavelets}
\subsection{Some helpful information about wavelets}
In this section we are using the notation $L^p$ for $L^p(\R)$. Let $\Psi$ be a wavelet. By this we mean that $\Psi\in L^2(\R)$ and functions $\{2^{j/2}\Psi(2^jx-k)\}_{(j, k)\in\Z^2}$ form orthonormal basis in $L^2(\R)$. We denote $2^{j/2}\Psi(2^jx-k)$ by 
$\Psi_{jk}(x)$.

Paper \cite{Wojt} contains a condition on $\Psi$ which guarantees that $\{\Psi_{jk}\}$ is unconditional basis not only for $L^2$ but for all $L^p$, $1<p<\infty$. Specifically, it says that there exists a function $\phi$ on $\R$ such that the following conditions hold for it:\\
1) $\phi(x)=\phi(-x)$ for all $x\in\R$;\\
2) $\phi$ is a decaying function on $[0, \infty]$;\\
3) $\phi$ is a bounded function on $\R$;\\
4) $\int_0^\infty \phi(x) \log(1+x)<\infty$;\\
5) $|\Psi(x)|\leq\phi(x)$ for all $x\in\R$.

We are assuming that this condition holds. It implies that $\{\Psi_{jk}\}$ is an unconditional basis for $L^p$. The proof of this fact is also presented in a book \cite{NPS}. Its main ingredient is a decomposition of a function into a sum of two other functions which we are going to need. In order to present it we will use some convenient notations.

Let $\eps=\{\eps_{jk}\}_{j,k\in\Z}$ be a collection of numbers each of which equals to $\pm 1$. We introduce the following operator $U_\eps$:
$$
U_\eps f:=\sum_{j, k\in\Z} \eps_{jk} \langle f, \Psi_{jk} \rangle \Psi_{jk}.
$$
In paper \cite{Wojt} it is proved that these operators are continuous in $L^p$ for all $1<p<\infty$ and their norms are uniformly bounded (in $\eps$). In fact, they are operators of weak type $(1, 1)$ with a constant which does not depend on $\eps$. We note that all of the subsequent facts are also true for operators $T$ of the form $(\Id+U_\eps)/2$ which are simply orthogonal projections in $L^2$ on $\spn\{\Psi_{jk}:(j, k)\in A\}$ where $A$ can be any subset of $\Z^2$ (by $\spn$ we mean the closed linear span).

For integer numbers $r$ and $l$ we denote the dyadic inerval $[2^{-r}l, 2^{-r}(l+1)]$ by $I_{rl}$.

For a function $f\in L^1$ and number $\lambda>0$ using Calder\' on--Zygmund decomposition we get the collection of intervals 
$\{I_{rl}\}_{(r,l)\in S}$ wih nonintersecting interiors such that for all of these intervals the following inequalities hold:
$$
\lambda< \frac{1}{|I_{rl}|}\int_{I_{rl}}|f|\leq 2\lambda
$$
and if $x\not\in \cup_{(r,l)\in S} I_{rl}$ the inequality $|f(x)|\leq\lambda$ holds a.e. We set $f_{rl}:=f\chi_{I_{rl}}$, 
$F:=\R\setminus \cup_{(r,l)\in S} I_{rl}$. Finally, we denote by $P_j$ the following orthogonal projection in $L^2$:
$$
P_jh:=\sum_{i<j}\sum_{k\in\Z} \langle h, \Psi_{ik}\rangle \Psi_{ik}
$$
and by $Q_j$ --- projection $\Id-P_j$:
$$
Q_jh:=\sum_{i\geq j}\sum_{k\in\Z} \langle h, \Psi_{ik}\rangle \Psi_{ik}
$$

The "good" part of the decomposition from paper \cite{Wojt} is then the function
$$
f_\lambda:=f\cdot\chi_f + \sum_{(r,l)\in S} P_r(f_{rl}).
$$
The remaining "bad" part is
$$
f-f_\lambda=\sum_{(r,l)\in S} Q_r(f_{rl}).
$$

We are going to need the following statements about this decomposition which are proved in the book \cite{NPS}.

\begin{fact}
Let $f$ be a function whith $\supp f\subset I_{rl}$. Then there exists a bounded even integrable function $\beta$ decaying on $[0, \infty)$ (and not depending on $f$) such that $\beta(2^j x)\leq 2^{4-j} \beta(x)$ if $|x|\geq 1$ and $j\in\Z_+$ and such that the following inequality holds:
$$
|P_rf(x)|\leq 2^r\|f\|_{L^1}\beta(2^rx-l).
$$
\end{fact}

\begin{fact}
Let $f$ be a function whith $\supp f\subset I_{rl}$. Then there exists an even integrable function $\eta$ that decays on the interval $[10, \infty)$ and such that if $|2^rx-l|>10$ the following inequality holds:
$$
|U_\eps Q_r f(x)|\leq \|f\|_{L^1} 2^r \eta(2^rx-l).
$$
Here $\eta$ does not depend on $f$ and $\eps$.
\end{fact}

We are going to need the following lemma which says that we can control $L^p$ norm of function $f_\lambda$.

\begin{lemma}
For any $1\leq p<\infty$ and any function $f\in L^1$ the following inequality holds:
$$
\Big\|\sum_{(r,l)\in S}P_r(f_{rl})\Big\|_{L^p}\lesssim \lambda^{1-1/p}\|f\|_{L^1}^{1/p}.
$$
\end{lemma}

We note that in \cite{NPS} and \cite{Wojt} this statement is proved only for $p=2$. However, our proof will be much like the proof in book \cite{NPS}.

\begin{proof}
At first we note that it is enough to prove the statement of lemma for integer values of $p$ --- in this case we can derive the required bound using interpolation (or simply H\" older's inequality). Thus we need to prove the inequality
$$
\int_{\R}\Big| \sum_{(r,l)\in S}P_r(f_{rl}) \Big|^{p}\lesssim \lambda^{p-1}\|f\|_{L^1}
$$
where $p$ is an integer not less than 1.

According to fact 1, the left hand side of this inequality does not exceed
\begin{align*}
  &\int_{\R}\Big| \sum_{(r,l)\in S} 2^r \|f_{rl}\|_{L^1}\beta(2^rx-l)  \Big|^p dx\\
  &\lesssim \sum_{(r_1,l_1)\in S}2^{r_1}\|f_{r_1l_1}\|_{L^1}\int_{\R}\beta(2^{r_1}x-l_1)\Big| \sum_{(r,l)\in S,r\geq r_1}2^r
  \|f_{rl}\|_{L^1}\beta(2^rx-l) \Big|^{p-1}dx.
\end{align*} 
 Using the fact that $2^r\|f_{rl}\|_{L^1}\leq 2\lambda$ we see that this expression is bounded by the following:
 $$
 \lambda^{p-1}\sum_{(r_1,l_1)\in S}\|f_{r_1l_1}\|_{L^1}2^{r_1}\int_{\R}\beta(2^{r_1}x-l_1)
 \Big| \sum_{(r,l)\in S,r\geq r_1}\beta(2^rx-l) \Big|^{p-1} dx.
 $$
 Changing the variable in the integral, we can write this expression in the following way:
$$
 \lambda^{p-1}\sum_{(r_1,l_1)\in S}\|f_{r_1l_1}\|_{L^1} \int_{\R}\beta(t)
 \Big| \sum_{(r,l)\in S, r\geq r_1} \beta(2^{r-r_1}t-(l-2^{r-r_1}l_1)) \Big|^{p-1}dt.
$$
For any fixed pair $(r_1, l_1)\in S$ we denote by $S'$ the set of pairs $\{(r-r_1, l-2^{r-r_1}l_1) : (r, l)\in S\}$. It is easy to see that $\{I_{rl}\}_{(r,l)\in S'}$ are also dyadic intervals with nonintersecting interiors. So we need to estimate the following expression:
$$
 \lambda^{p-1}\sum_{(r_1,l_1)\in S}\|f_{r_1l_1}\|_{L^1} \int_{\R}\beta(t)
 \Big| \sum_{(r,l)\in S', r\geq 0} \beta(2^rt-l)\Big|^{p-1} dt.
$$
Now we prove that the integral in this expression is bounded by constant which does not depend on $S'$. Clearly the statement of lemma will follow immediately. So it is left to prove that for every $k\in\Z_+$ the following inuequality holds with constant 
$C$ depending on $k$ but not on $S'$:
$$
 \int_{\R} \beta(t)\Big( \sum_{(r,l)\in S', r\geq 0}\beta(2^rt-l) \Big)^k dt\leq C.
$$
We prove this by induction in $k$. The inequality is obvious for $k=0$ since $\beta$ is an integrable function. Now assume this inequality holds for $k-1$ and we prove that it holds also for $k$. Note that
\begin{align*}
 &\int_{\R} \beta(t)\Big( \sum_{(r,l)\in S', r\geq 0}\beta(2^rt-l) \Big)^k dt\\
 &\lesssim\sum_{(r,l)\in S', r\geq 0} \int_{\R}\beta(t)\beta(2^rt-l)\Big(\sum_{(r_1,l_1)\in S', r_1\geq r}
 \beta(2^{r_1}t-l_1)\Big)^{k-1}dt.
\end{align*}
Let us denote by $S_{nr}$ the set $\{l: (r,l)\in S', I_{rl}\subset [n,n+1]\}$ and by $\varkappa_{nr}$ --- the cardinality of $S_{nr}$. Since $r\geq 0$, every interval $I_{rl}$ is contained in the interval of the form $[n, n+1]$ with integer $n$, so we can rewrite our expression in the following way:
$$
\sum_{n\in\Z}\sum_{r=0}^{\infty}\sum_{l\in S_{nr}}\int_{\R}
 \beta(t)\beta(2^rt-l)\Big( \sum_{(r_1,l_1)\in S', r_1\geq r} \beta(2^{r_1}t-l_1) \Big)^{k-1}dt.
$$
Now for any integer $n$ we can split our integral into three parts:
\begin{align*}
 J_{n1}&:=\sum_{r=0}^{\infty}\sum_{l\in S_{nr}}\int_{n-10}^{n+10}
 \beta(t)\beta(2^rt-l)\Big( \sum_{(r_1,l_1)\in S', r_1\geq r} \beta(2^{r_1}t-l_1) \Big)^{k-1}dt,\\
 J_{n2}&:=\sum_{r=0}^{\infty}\sum_{l\in S_{nr}}\int_{-\infty}^{n-10}
 \beta(t)\beta(2^rt-l)\Big( \sum_{(r_1,l_1)\in S', r_1\geq r} \beta(2^{r_1}t-l_1) \Big)^{k-1}dt,\\
 J_{n3}&:=\sum_{r=0}^{\infty}\sum_{l\in S_{nr}}\int_{n+10}^{+\infty}
 \beta(t)\beta(2^rt-l)\Big( \sum_{(r_1,l_1)\in S', r_1\geq r} \beta(2^{r_1}t-l_1) \Big)^{k-1}dt.
\end{align*}
Now we estimate each of these terms separately. We start with $J_{n1}$:
\begin{align*}
 J_{n1}&\leq(\max_{[n-10,n+10]}\beta) \sum_{r=0}^{\infty}\sum_{l\in S_{nr}}\int_{\R}
 \beta(2^rt-l)\Big( \sum_{(r_1,l_1)\in S', r_1\geq r} \beta(2^{r_1}t-l_1) \Big)^{k-1} dt \\
 &=(\max_{[n-10,n+10]}\beta) \sum_{r=0}^{\infty}\sum_{l\in S_{nr}}2^{-r}\int_{\R}
 \beta(t)\Big( \sum_{(r_2,l_2)\in S'', r_2\geq 0} \beta(2^{r_2}t-l_2) \Big)^{k-1}dt.
\end{align*}
In order to pass to the last line we used the change of variable which we have already done before. Here $S''$ is the set of pairs of integers depending on $(r, l)$ but it is true for it that $\{I_{r_2l_2}\}_{(r_2, l_2)\in S''}$ are non-intersecting intervals.  Using induction hypothesis we conclude that the integral in the expression does not exceed some constant which does not depend on $(r,l)$ and thus our expression is less than or equal to 
$$
C\max_{[n-10,n+10]}\beta\sum_{r=0}^{\infty}2^{-r}\varkappa_{nr}.
$$
Here $\sum_{r=0}^{\infty}2^{-r}\varkappa_{nr}$ is the sum of lengths of nonintersecting intervals contained in $[n, n+1]$ and so it does not exceed 1. We conclude that
$$
J_{n1}\lesssim\max_{[n-10,n+10]}\beta.
$$
Using the fact that $\beta$ is a decaying on $[0, +\infty]$ even integrable function we conclude:
$$
\sum_{n\in\Z}J_{n1}\lesssim\sum_{n\in\Z}\max_{[n-10,n+10]}\beta\leq C.
$$

Now we estimate $J_{n2}$. If $l\in S_{nr}$, then $I_{rl}\subset [n, n+1]$ and so $2^{-r}l \geq n$. So if $t< n-10$ then
$2^rt-l=2^r(t-2^{-r}l)\leq 2^r(t-n)<0$. Using the properties of $\beta$ from fact 1 we can conclude that the following inequality holds:
\begin{align*}
 J_{n2}&\leq \sum_{r=0}^{\infty}\sum_{l\in S_{nr}} \int_{-\infty}^{n-10}
 \beta(t)\beta(2^{r}(t-n))\Big( \sum_{(r_1,l_1)\in S', r_1\geq 0} \beta(2^{r_1}t-l_1) \Big)^{k-1}dt\\
 &\lesssim\sum_{r=0}^{\infty}2^{-r}\varkappa_{nr}\int_{-\infty}^{n-10}
 \beta(t)\beta(t-n)\Big( \sum_{(r_1,l_1)\in S', r_1\geq 0} \beta(2^{r_1}t-l_1) \Big)^{k-1}dt.
\end{align*}
As we already mentioned, $\sum_{r=0}^{\infty}2^{-r}\varkappa_{nr}\leq 1$. Then, using monotonicity and integrability of function $\beta$ we see that $\sum_{n\in\Z}\beta(t-n)$ is a uniformly bounded function and we get the following estimate:
$$
\sum_{n\in\Z}J_{n2}\lesssim \int_{\R}\beta(t) \Big( \sum_{(r_1,l_1)\in S', r_1\geq 0}\beta(2^{r_1}t-l_1) \Big)^{k-1}dt.
$$
Using the induction hypothesis we see that the right hand side is bounded by some constant. The term $\sum_{n\in\Z}J_{n3}$ is estimated in exactly the same way --- if $t\geq n+10$ and $I_{rl}\subset [n, n+1]$, then $2^{-r}l\leq n+1$ and 
$2^rt-l=2^r(t-2^{-r}l)\geq 2^r(t-n-1)>0$ and the estimates similar to that we have done above show that $\sum_{n\in\Z}J_{n3}\leq C$ and the lemma is proved.
\end{proof}

Clearly, since on the set $F$ the inequality $|f|\leq\lambda$ holds, the lemma we just proved implies the inequality
$$
\|f_\lambda\|_{L^p}\lesssim \lambda^{1-1/p} \|f\|_{L^1}^{1/p}.
$$

\subsection{Stability theorem for couple $(L^1, L^p)$}
Now we pass to the proof of the stability theorem. Here $T$ will denote the projection on $\spn\{\Psi_{jk}: (j,k)\in A\}$ described previously although any operator bounded on $L^p$ and for which fact 2 holds would suit us (every such operator is of weak type $(1,1)$). In this situation the analogue of theorem 1 is true.

\begin{theorem}
Let $T$ be as above, $1<p<\infty$ and $f\in L^1$ is a function for which $Tf\in L^1$. Then for any $s>0$ there exists
such function $u^{(s)}\in L^1$ that the following conditions hold:
\begin{align}
   \|u^{(s)}\|_{L^p}\lesssim& s, \label{1st}\\
   \|f-u^{(s)}\|_{L^1}\lesssim& \dist_{L^1}(f,B_{L^p}(s)), \label{2nd}\\
   \|Tf-Tu^{(s)}\|_{L^1}\lesssim& \dist_{L^1}(f,B_{L^p}(s))+\dist_{L^1}(Tf,B_{L^p}(s)). \label{3rd}
  \end{align}
\end{theorem}

\begin{proof}
Let $h$ be any near-minimizer such that $\|h\|_{L^p}\leq s$ and $\|f-h\|_{L^1}\leq 2\dist_{L^1}(f,B_{L^p}(s))$. Then we set $u^{(s)}:=h+(f-h)_t$ where $t$ satisfies the condition $t^{p-1}\|f-h\|_{L^1}=s^p$. We remind the reader that here by $(f-h)_t$ we understand the "good" part of the decomposition which is described previously applied to the function $f-h$ and the number $t$. Now we check that $u^{(s)}$ is also a near-minimizer which means that conditions \eqref{1st} and \eqref{2nd} hold for it. The inequality \eqref{2nd} follows immediately from the fact that according to the lemma we proved $\|(f-h)_t\|_{L^1}\lesssim \|f-h\|_{L^1}$. In order to prove the condition \eqref{1st}, it is enough to check that $\|(f-h)_t\|_{L^p}\lesssim s$. But using our choise of $t$ and lemma 1 once again we can write: $\|(f-h)_t\|_{L^p}\lesssim t^{1-1/p}\|f-h\|_{L^1}^{1/p}=s$.

It is left to check the condition \eqref{3rd}. In order to do it we choose a function $v\in L^1$ which is a near-minimizer for $Tf$: $\|v\|_{L^p}\leq s$, $\|Tf-v\|_{L^1}\leq 2\dist_{L^1}(Tf, B_{L^p}(s))$. Let $\{I_{rl}\}_{(r,l)\in S}$ be the set of dyadic intervals arising in the construction of function $(f-h)_t$. Then the following estimate holds:
$$
 \sum_{(r,l)\in S}|I_{rl}|\leq t^{-1}\|f-h\|_{L^1}=\Big(\frac{\|f-h\|_{L^1}}{s}\Big)^{p'}
 \lesssim \Big(\frac{\dist_{L^1}(f, B_{L^p}(s))}{s}\Big)^{p'}.
$$
Here $p'=\frac{p}{p-1}$. Now we write:
\begin{equation}
 \|Tf-Tu^{(s)}\|_{L^1}\leq \int_{\R\setminus\cup 30 I_{rl}} |Tf-Tu^{(s)}| + \int_{\cup 30 I_{rl}}|Tf-v| +
 \int_{\cup 30 I_{rl}}|Tu^{(s)}-v|.
 \label{eq1}
 \end{equation}
 Let us estimate the first summand. We note that it can be written in the following way:
 \begin{align*}
 &\int_{\R\setminus\cup 30 I_{rl}} |Tf-Tu^{(s)}|=\int_{\R\setminus\cup 30 I_{rl}}|T((f-h)-(f-h)_t)|\\
 =&\int_{\R\setminus\cup 30 I_{rl}} \Big|T \Big(\sum_{(r,l)\in S}Q_r((f-h)_{rl})\Big)\Big|\leq
 \sum_{(r,l)\in S}\int_{\R\setminus 30 I_{rl}} |T(Q_r((f-h)_{rl}))|dx.
 \end{align*}
 According to fact 2 this expression can be bounded by the following:
 \begin{align*}
 &\sum_{(r,l)\in S}\int_{\R\setminus 30 I_{rl}} \|(f-h)_{rl}\|_{L^1} 2^r\eta(2^rx-l)dx\\
 \leq &\sum_{(r,l)\in S} \|(f-h)_{rl}\|_{L^1} \int_\R\eta(x) dx\lesssim\|f-h\|_{L^1}.
 \end{align*}
Due to our choice of $h$ this expsession is less than or equal to $2\dist_{L^1}(f,B_{L^p}(s))$.

The second summand in \eqref{eq1} is obviously less than or equal to $\|Tf-v\|_{L^1}\leq 2\dist_{L^1}(Tf,B_{L^p}(s))$. In order to estimate the third one we use the H\" older's inequality and conclude that it does not exceed the following expresion:
$$
 \|Tu^{(s)}-v\|_{L^p}\Big(\sum_{(r,l)\in S}|30 I_{rl}|\Big)^{1/p'}\lesssim
 (\|Tu^{(s)}\|_{L^p}+\|v\|_{L^p})\Big(\frac{\|f-h\|_{L^1}}{t}\Big)^{1/p'}.
$$
Using the boundedness of $T$ on $L^p$ we get that the third summand in the right hand side of the inequality \eqref{eq1} is estimated by
$$
 s\Big(\frac{\|f-h\|_{L^1}}{t}\Big)^{1/p'}=\|f-h\|_{L^1}\leq 2\dist_{L^1}(f,B_{L^p}(s)).
$$
So we checked that the property \eqref{3rd} holds and the theorem is proved.
\end{proof}

Now we turn to some corollaries of the theorem we just proved.
\begin{corollary}
 Suppose $1<p<\infty$, $T$ is an operator from the theorem and $f\in L^1$ is a function for which $Tf\in L^1$. Then there exists a sequence of functions $f_k\in L^1\cap L^p$ tending to $f$ in $L^1$ for which $Tf_k\in L^1$ and $\|Tf_k-Tf\|_{L^1}\rightarrow 0$.
\end{corollary}
\begin{proof}
This statement immediately follows from the theorem if we tend $s$ to infinity (in this case since $L^1\cap L^p$ is dense in $L^1$ the right hand sides of inequalities \eqref{2nd} and \eqref{3rd} tend to zero).
\end{proof}

We note that if $T$ is a projection described previously and $E$ is a measurable subset of $\R$ then $\chi_E T$ is of course a bounded operator on $L^p$ and fact 2 holds for it. So we have the following generalization of the previous corollary.
\begin{corollary}
Suppose $1<p<\infty$, $T$ is an operator from the theorem, $f$ is a function from $L^1$ and set $E\subset\R$ is such that 
$\chi_E Tf\in L^1$. Then there exist functions $f_k\in L^1\cap L^p$ tending to $f$ in $L^1$ for which $\chi_E Tf_k\in L^1$ and 
$\chi_E Tf_k\rightarrow \chi_E Tf$ in $L^1$.
\end{corollary}
\begin{proof}
It is enough to use the theorem for operator $\chi_E T$ and then repeat the proof of the previous corollary.
\end{proof}

Using the first corollary it is easy to see that if a function from $L^1$ has some of the wavelet coefficients equal to zero then it can be approximated by functions from $L^1\cap L^p$ for which the same coefficients are also zero. Here is the precise statement of this fact.
\begin{corollary}
Suppose $1<p<\infty$ and $f\in L^1$. Then there exist functions $g_k\in L^1\cap L^p$ tending to $f$ in $L^1$ such that if $\langle f, \Psi_{rl}\rangle = 0$ then $\langle g_k, \Psi_{rl}\rangle = 0$.
\end{corollary}
\begin{proof}
Denote by $A$ the set  $\{(r,l):\langle f,\Psi_{rl} \rangle \neq 0\}$ and let $T$ be the orthogonal projection on $\spn\{\Psi_{rl}:(r,l)\in A\}$. Then $Tf=f$ and we can set $g_k=Tf_k$ where $f_k$ are the functions from the first corollary.
\end{proof}

\section{Weighted stability for singular integrals}

In this section we will be interested in weighted spaces $L^p(\R^d; w)$ and action of singular integral operators on them. The standart information about these things can be found for example in the book \cite{RdF}. By singular integral operator (or Calder\' on--Zygmund operator) we mean the operator $T$ bounded on $L^2(\R^d)$ and which possesses the kernel --- the function 
$K(x, y)$ such that
$$
(Tf)(x)=\int_{\R^d}K(x,y)f(y) dy
$$
for all $f$ with compact support and all $x$ outside this support. We assume that for the kernel $K$ and $x$, $y_1$, $y_2$ such that $y_1$ and $y_2$ are inside some cube $Q$ and $x\not \in 5Q$ the following inequality holds:
$$
|K(x,y_1)-K(x,y_2)|\leq C \frac{|y_1-y_2|^\alpha}{|x-y_1|^{d+\alpha}},
$$
where $\alpha$ is a positive number (not depending on $x$, $y_1$ and $y_2$). Besides that, we will need weights from
Muckenhoupt classes $A_p$ --- all necessary information about them (in particular, the boundedness of Calderon – Zygmund operators
on the spaces $L^p(w)$ with $w\in A_p$) can be found in the books \cite{RdF} and \cite{Graf}. In the book \cite{RdF}, among other things,
the following fact is proved, which is a weight analogue of the property of Calder\' on – Zygmund operators, called in \cite{KK}
long-range $L^1$ -regularity:
\begin{fact}
 Let $T$ be a Calder\' on--Zygmund operator and $f$ be a function with support in cube $Q$ such that $\int f=0$, $w\in A_1$. Then 
 $$
 \int_{\R^d\setminus 2\sqrt{d}Q}|Tf(x)|w(x)\lesssim \int_{\R^d}|f(x)|w(x) dx.
 $$
\end{fact}

We note that, strictly speaking, in \cite{RdF} only singular integrals of convolution type are considered, that is, for
which $K(x, y)$ depends only on $ x-y $. However this plays no role in the proofs of the statements we need
(in particular, fact 3).

So, our goal is to prove an analogue of Theorem 1 for spaces with weights. To do this, we use the analogue of
Calder\' on--Zygmund decomposition which can be found in the article \cite{AK}. For an arbitrary weight $w$ and a measurable set $E$, we will
use the standart notation $w(E)$ for $\int_E w$. Suppose $a\in A_\infty$, $w\in A_1$, $ G\in L^1(w)$.
We set $b = \frac{a}{w}$, $g = Gb^{-1}$. Then $g\in L^1 (a)$. The weight $a$ lying in $ A_\infty$ possesses the doubling condition
(that is, $ a(2Q) \lesssim a(Q) $ for any cube $Q$), and therefore a Calder\' on--Zygmund partition can be applied to $g$ with weight $a$ and the parameter $\lambda$ and we get a set of non-intersecting dyadic cubes $ \{Q_i \} $, such that
$$
\lambda\leq\frac{1}{a(Q_i)}\int_{Q_i} |Gb^{-1}|a\leq C\lambda,
$$
and $|Gb^{-1}|\leq\lambda$ almost everywhere outside $\cup Q_i$. Then the "good part" of the decomposition is the function $G_t$, defined as follows:
$$
G_\lambda(x)=
\begin{cases}
 G(x), x\not \in\cup Q_i,\\
 \frac{b(x)}{b(Q_i)}\int_{Q_i} G, x\in Q_i.
\end{cases}
$$
We present the properties of this decomposition; their prooves can be found in the article \cite{AK}. We denote by $ \tilde{Q}$ the
cube $ 2\sqrt{d}Q $.
\begin{fact}
The cubes $Q_i$ we presented and the function $G_\lambda$ possess the following properties:\\
 1) $|G_\lambda|\lesssim \lambda b$;\\
 2) $\|G_\lambda\|_{L^1(w)}\lesssim \|G\|_{L^1(w)}$ and thus $\|G-G_\lambda\|_{L^1(w)}\lesssim\|G\|_{L^1(w)}$;\\
 3) $\int_{Q_i}(G-G_\lambda)=0$;\\
 4) $a(Q_i)\leq\frac{1}{\lambda}\int_{Q_i}|G|w$, and so $a(\cup \tilde{Q_i})\lesssim \frac{1}{\lambda}\|G\|_{L^1(w)}$.
\end{fact}

We now pass to the stability theorem.
\begin{theorem}
 Suppose $ 1 <p <\infty $ and let the weights $ w $ and $ v $ be such that $ w \in A_1 $, $ v \in A_p $ and
  $ a:= (\frac{w^p}{v})^{\frac{1}{p-1}} \in A_\infty $. Suppose
  $ T $ is the Calder\' on--Zygmund operator and the function $ f \in L^1(w) $ is such that $ Tf \in L^1(w)$. Then for any $s>0$                        there exists
  a function $ u^{(s)} \in L^1(w) $ such that
  \begin{align*}
  \|u^{(s)}\|_{L^p(v)}\lesssim& s,\\
  \|f-u^{(s)}\|_{L^1(w)}\lesssim& \dist_{L^1(w)}(f, B_{L^p(v)}(s)),\\
  \|Tf-Tu^{(s)}\|_{L^1(w)}\lesssim& \dist_{L^1(w)}(f,B_{L^p(v)}(s))+\dist_{L^1(w)}(Tf, B_{L^p(v)}(s)).
 \end{align*}
\end{theorem}
\begin{proof}
 Once the decomposition we need is described, for the proof of the theorem it remains only to repeat the argument from the book \cite{KK}.
 Let $ h $ be a function for which the inequalities $\|h\|_{L^p(v)}\leq s$, 
 $\|f-h\|_{L^1(w)}\leq 2\dist_{L^1(w)}(f,B_{L^p(v)}(s))$ hold. We set $ u^{(s)}:= h + (f-h)_t$ where $ t $ is a number such that
 $ t^{p-1} \| f-h \|_{L^1(w)} = s^p $. Here $ (f-h)_t $ is the function described above
 (and it was constructed with respect to the weights $ w \in A_1 $ and $ a \in A_\infty $). We check that
 $ u^{(s)} $ is a near-minimizer. Indeed, 
 $$
 \|f-u^{(s)}\|_{L^1(w)}\leq \|f-h\|_{L^1(w)}+\|(f-h)_t\|_{L^1(w)}
 $$
 which by the fact 4 does not exceed 
 $$
 C\|f-h\|_{L^1(w)}\lesssim \dist_{L^1(w)}(f,B_{L^p(v)}(s))
 $$
 The norm of $u^{(s)}$ in $ L^p(v) $ is also easily estimated:
 $$
 \|u^{(s)}\|_{L^p(v)}\leq \|h\|_{L^p(v)}+\|(f-h)_t\|_{L^p(v)}\leq s+\Big(\int|(f-h)_t|v\Big)^{1/p}.
 $$
 Taking into consideration that, according to fact 4, $|(f-h)_t|\lesssim tb$, where $b=aw^{-1}=(wv^{-1})^{\frac{1}{p-1}}$,
  the second term, up to a constant multiplication, is less than or equal to
  $$
 \Big(\int t^{p-1} b^{p-1} |(f-h)_t| v\Big)^{1/p}=t^{\frac{p-1}{p}}\|(f-h)_t\|_{L^1(w)}\lesssim
 t^{\frac{p-1}{p}}\|(f-h)\|_{L^1(w)}=s.
 $$
 Thus $\|u^{(s)}\|_{L^p(v)}\lesssim s$. It remains to check that the last propert holds, that is, the stability of $u^{(s)}$ under the action of $T$. In order to do this, we consider the near-minimizer $g$ for $Tf$ such that 
 $\|g\|_{L^p(v)}\leq s$ и
 $\|Tf-g\|_{L^1(w)}\leq 2\dist_{L^1(w)}(Tf, B_{L^p(v)}(s))$ and write: 
 $$
 \|T(f-u^{(s)})\|_{L^1(w)}\leq \int_{\R^d\setminus\cup\tilde{Q_i}} |Tf-Tu^{(s)}|w + 
 \int_{\cup\tilde{Q_i}}|Tf-g|w+\int_{\cup\tilde{Q_i}}|Tu^{(s)}-g|w.
 $$
\end{proof}
We estimate the first term. Note that $f-u^{(s)}=(f-h)-(f-h)_t$ is a function with support in $\cup Q_i$, moreover, according to fact 4, its integral over each of the cubes $ Q_i $ is equal to zero. Therefore, using fact 3, the first summand can be estimated by
$$
\|(f-h)-(f-h)_t\|_{L^1(w)}\lesssim \|f-h\|_{L^1(w)}\lesssim \dist_{L^1(w)}(f, B_{L^p(v)}(s)).
$$
The second summand is less than or equal to
$$
\|Tf-g\|_{L^1(w)}\leq 2\dist_{L^1(w)}(Tf, B_{L^p(v)}(s)).
$$
In order to estimate the third one we use the H\" older's inequality:
$$
\int_{\cup\tilde{Q_i}}|Tu^{(s)}-g|w=\int_{\cup\tilde{Q_i}}|Tu^{(s)}-g|v^{1/p}a^{1/p'}\leq
\Big(\int_{\cup\tilde{Q_i}}|Tu^{(s)}-g|^p v\Big)^{1/p}a(\cup\tilde{Q_i})^{1/p'}.
$$
Finally, using the last statement of fact 4 (as well as the facts that the operator $ T $ is bounded on $L^p(v)$
and $ \| g \|_{L^p(v)} \lesssim s $, $ \| u^{(s)} \|_{L^p(v)} \lesssim s $),
we conclude that our expression is estimated by the following:
$$
(\|Tu^{(s)}\|_{L^p(v)}+\|g\|_{L^p(v)})\frac{1}{t^{1/p'}}\|f-h\|_{L^1(w)}^{1/p'}\lesssim
s\Big(\frac{\|f-h\|_{L^1(w)}}{t}\Big)^{\frac{p-1}{p}}=\|f-h\|_{L^1(w)}.
$$
According to our choise of the function $h$, this expression is less than or equal to
$$
2\dist_{L^1(w)}(f,B_{L^p(v)}(s)).
$$
It remains to collect the estimates and the theorem is proved.

\end{document}